\documentclass[11pt]{article}
\usepackage{mathrsfs}
\usepackage{amsmath}
\usepackage{amsfonts}
\usepackage{amssymb}
\usepackage{latexsym}
\usepackage{enumerate}
\usepackage[dvips]{graphicx}
 \usepackage{graphicx}
\newtheorem{theorem}{\bf Theorem}[section]
\newtheorem{lemma}[theorem]{\bf Lemma}

\newtheorem{corollary}[theorem]{\bf Corollary}

\newenvironment{proof}{\noindent{\em Proof:}}{\quad \hfill$\Box$\vspace{2ex}}

\newtheorem{definition}[theorem]{\bf Definition}

\def \bI {\Bbb I}

\def \bR {\Bbb R}

\def \bB {\Bbb B}
\def \bC {\Bbb C}

\def \cR {{\cal R}}

\def \and {\,\mbox{\rm and}\,}

\def \supp {\,{\rm supp}\,}

\def \sgn {\,{\rm sgn}\,}

\def \ae {\,{\rm a.e.}\,}
\def \conv {\,{\rm conv}\,}
\makeatletter

\newcommand{\Rmnum}[1]{\expandafter\@slowromancap\romannumeral #1@}
\makeatother

\begin{document}

\title{Existence of the Bedrosian Identity for Singular Integral Operators}

\author{Rongrong Lin\thanks{School of Mathematics and Computational
Science, Sun Yat-sen University, Guangzhou 510275, P. R. China. E-mail address: {\it linrr@mail2.sysu.edu.cn}.}\quad and\quad Haizhang Zhang\thanks{Corresponding author. School of Mathematics and Computational
Science and Guangdong Province Key Laboratory of Computational
Science, Sun Yat-sen University, Guangzhou 510275, P. R. China. E-mail address: {\it zhhaizh2@sysu.edu.cn}. Supported in part by Natural Science Foundation of China under grants 11222103 and 11101438.}}
 \date{}
\maketitle

%\subjclass[2000]{Primary 46C15; Secondary 46E22}
%
%\date{}
%
%\dedicatory{}

\begin{abstract}
The Hilbert transform $H$ satisfies the Bedrosian identity $H(fg)=fHg$ whenever the supports of the Fourier transforms of $f,g\in L^2(\bR)$ are respectively contained in $A=[-a,b]$ and $B=\bR\setminus(-b,a)$, $0\le a,b\le+\infty$. Attracted by this interesting result arising from the time-frequency analysis, we investigate the existence of such an identity for a general bounded singular integral operator on $L^2(\bR^d)$ and for general support sets $A$ and $B$. A geometric characterization of the support sets for the existence of the Bedrosian identity is established. Moreover, the support sets for the partial Hilbert transforms are all found. In particular, for the Hilbert transform to satisfy the Bedrosian identity, the support sets must be given as above.

\noindent{\bf Keywords:} the Bedrosian identity, singular integral operators, translation-invariant operators, the Hilbert transform, the Riesz transforms
\end{abstract}

\emph{}
\section{Introduction}
\setcounter{equation}{0}

This note is motivated by the important Bedrosian identity in the time-frequency analysis, \cite{Bedrosian,Brown74,Brown86,Cohen}. The analytic signal method is a key tool in extracting the frequency information of a given signal, \cite{Picibono}. The method engages the Hilbert transform to form the imaginary part of a real signal. We recall that the Hilbert transform $H$ is a bounded linear operator on $L^2(\bR)$ defined by
$$
(Hf)\hat{\,}(\xi)=-i\sgn(\xi)\hat{f}(\xi),\ \ \xi\in\bR,\ f\in L^2(\bR),
$$
where $\sgn(\xi)$ takes the value $-1,0,1$ for $\xi<0$, $\xi=0$, and $\xi>0$, respectively. The Fourier transform adopted in this note has the form
$$
\hat{f}(\xi):=\frac1{(2\pi)^{d/2}}\int_{\bR^d} f(x)e^{-ix\cdot\xi}dx,\ \ \xi\in\bR^d
$$
for $f\in L^1(\bR^d)$. The Hilbert transform is also defined on $L^p(\bR)$ for $1\le p<+\infty$ as the singular integral operator with the convolution kernel p.v. $\frac1{\pi x}$, \cite{Butzer}. Note that the Hilbert transform is {\it translation-invariant} in the sense that
$H(\tau_xf)=\tau_x(Hf)$ for all $f\in L^2(\bR)$ and $x\in\bR$. Here, $\tau_xf:=f(\cdot-x)$.

The Bedrosian theorem, established in the time-frequency analysis \cite{Bedrosian,Brown74,Brown86,Nuttall}, states that the identity
\begin{equation}\label{Bedrosian}
H(fg)=fHg
\end{equation}
holds true if the supports of the Fourier transforms of $f,g\in L^2(\bR)$ satisfy $\supp\hat{f}\subseteq[-a,b]$ and $\supp\hat{g}\subseteq\bR\setminus(-b,a)$ for some nonnegative constants $0\le a,b\le+\infty$.

We call (\ref{Bedrosian}) the {\it Bedrosian identity}. The functions $f,g$ in (\ref{Bedrosian}) can be naturally interpreted as the amplitude and phase functions of a given signal, respectively. This physical interpretation accounts for its usefulness in the time-frequency analysis \cite{Cohen} and latter in the empirical mode decomposition (EMD) \cite{Huang}. Especially, the success of EMD in signal analysis has stimulated much interest from the mathematics community. Various characterizations of the Bedrosian identity and reproofs of the Bedrosian theorem have been established (see, for example, \cite{Qian2009,Tan2010,Tan2009,Wang,Xu,YangandZhang,ZhangandYu}).

Fascinated by the Bedrosian identity, we shall investigate in the note the existence of such an identity for a general bounded singular integral operator on $L^2(\bR^d)$. To make our objective more precise, we introduce the following definition.

\begin{definition}
We say that a bounded linear operator $T:L^{2}(\bR^{d})\rightarrow L^{2}(\bR^{d})$ satisfies the Bedrosian identity with respect to a pair $(A,B)$ of Lebesgue measurable sets $A,B\subseteq\bR^d$ if the identity
\begin{equation}\label{bedrosianidentityT}
    T(fg)=f (Tg)
\end{equation}
 holds true for all  $f,g\in L^{2}(\bR^{d})\cap L^{\infty}(\bR^d)$ with
$$
\supp\hat{f}\subseteq A, \ \  \supp\hat{g}\subseteq B.
$$
\end{definition}

Using the above definition, we can restate the Bedrosian theorem as that the Hilbert transform $H$ satisfies the Bedrosian identity with respect to $([-a,b],\bR\setminus (-b,a))$ for any nonnegative constants $0\le a,b\le+\infty$.

Now, from the pure mathematics perspective, we ask the following three questions.
\begin{description}
\item[(A):] For what support sets $A,B$, does there exist a nontrivial bounded singular integral operator $T$ on $L^2(\bR^d)$ that satisfies the Bedrosian identity with respect to $(A,B)$?

\item[(B):] How to characterize such an operator?

\item[(C):] Conversely, given a class of bounded singular integral operators, can we find a pair of support sets $A,B$ so that each operator in the class satisfies the Bedrosian identity with respect to $(A,B)$? For instance, can the Riesz transforms or other bounded singular integral operators in harmonic analysis satisfy the Bedrosian identity with respect to some support sets?
\end{description}

In this note, we will be able to give a complete answer to the first two questions, and to the third one for the partial Hilbert transforms and for the Riesz transforms. To describe our results, we shall recall some definitions and known facts. Firstly, we consider singular integral operators $T$ of the form
$$
(Tf)(x):=\mbox{p.v. }\int_{\bR^d}f(y)K(x-y)dy,
$$
where $K$ is a prescribed integral kernel such that $T$ is bounded on $L^2(\bR^d)$. We note that such an operator $T$ is {\it translation-invariant}, that is, $T(\tau_x f)=\tau_x (Tf)$ for all $x\in\bR^d$ and $f\in L^2(\bR^d)$. It is well-known (see, for example, \cite{Grafakos}, page 140) that a bounded linear operator $T$ on $L^2(\bR^d)$ is translation-invariant if and only if there exists a function $m\in L^\infty(\bR^d)$ such that
\begin{equation}\label{fouriermultiplier}
(Tf)\hat{\,}=m\hat{f},\ \ f\in L^2(\bR^d).
\end{equation}
The function $m$ above is called the {\it Fourier multiplier} of the operator $T$. Secondly, we shall mainly restrict our discussion on support sets that are open. Some geometric treatments are needed.

\begin{definition} {\bf(Characteristic sets)} Let $A,B$ be two nonempty open subsets in $\bR^d$. We decompose $A,B$ and $A+B$ into unions of their corresponding connected components as follows:
\begin{equation}\label{components}
A=\bigcup_{i\in I} A_i,\ B=\bigcup_{j\in J} B_j,\ A+B=\bigcup_{k\in K} C_k,
\end{equation}
where $I,J,K$ are index sets. Also denote for each $k\in K$ by
\begin{equation}\label{components1}
J_k:=\{j\in J:A_i+B_j\subseteq C_k\mbox{ for some }i\in I\}
\end{equation}
and introduce
\begin{equation}\label{components2}
D_k:=C_k\cup\bigcup_{j\in J_k}B_j.
\end{equation}
We call these $D_k$ the {\it characteristic sets} of $(A,B)$ for the problem under investigation.
\end{definition}

The main results of this note are described as follows.

\begin{theorem}\label{main1}
Let $A,B$ be two nonempty open subsets in $\bR^d$ and let $D_k$, $k\in K$ be the characteristic sets of $(A,B)$ defined above. Suppose $T$ is a bounded linear operator on $L^2(\bR^d)$ given by (\ref{fouriermultiplier}) through a Fourier multiplier $m\in L^\infty(\bR^d)$. Then $T$ satisfies the Bedrosian identity with respect to $(A,B)$ if and only if there exist constants $c_k\in\bC$, $k\in K$ such that
$$
m=c_k\mbox{ almost everywhere on }D_k,\ \ k\in K.
$$
Here, if $D_k\cap D_{k'}$ is nonempty then $c_k=c_{k'}$.
\end{theorem}
The above theorem answers question {\bf (B)}. One can also tell from this result that for a bounded singular integral operator $T$ on $L^2(\bR^d)$ to satisfy the Bedrosian identity, its Fourier multiplier must be constant on a set of positive Lebesgue measure. Thus, when $d\ge2$, the Riesz transforms $\cR_j$'s defined by
$$
(\cR_jf){\hat{\,}}(\xi):=\frac{-i\xi_j}{\|\xi\|}\hat{f}(\xi), \ \ \xi\in\bR^d,\  f\in L^2(\bR^d),\ \ 1\le j\le d
$$
cannot satisfy a Bedrosian identity. Here, $\|\cdot\|$ denotes the standard Euclidean norm on $\bR^d$.

We are able to answer question {\bf (A)} as a direct consequence of Theorem \ref{main1}.

\begin{corollary}\label{main2}
There exists a nontrivial bounded linear translation-invariant operator on $L^2(\bR^d)$ that satisfies the Bedrosian identity with respect to a pair of nonempty open subsets $A,B\subseteq\bR^d$ if and only if for every $k\in K$, the set complement of $\cup\{D_{k'}:k'\in K, \ c_{k'}=c_k\}$ in $\bR^d$ has positive Lebesgue measure.
\end{corollary}

Finally, we answer question {\bf (C)} for the partial Hilbert transforms on $L^2(\bR^d)$. These operators are defined as follows:
\begin{equation}\label{partialHilbert}
(H_jf){\hat{\,}}(\xi):=-i\sgn(\xi_j)\hat{f}(\xi), \ \  \xi\in\bR^d,\  f\in L^2(\bR^d),\ 1\le j\le d.
\end{equation}
The partial Hilbert transforms turn out to be the ideal singular integral operators for the time-frequency of multi-dimensional signals, \cite{Bulow,Hahn}.

Denote for each Lebesgue measurable set $E\subseteq\bR^d$ by $\tilde{E}$ the set of all points $x_0\in \bR^d$ with the property that $U\cap E$ is of positive Lebesgue measure for any open neighborhood $U$ of $x_0$.

\begin{theorem}\label{main3}
Let $A,B$ be two Lebesgue measurable sets in $\bR^d$. Then all partial Hilbert transforms $H_j$, $1\le j\le d$ satisfy the Bedrosian identity with respect to $(A,B)$ if and only if there exist constants $0\le a_j,b_j\le +\infty$, $1\le j\le d$ such that
\begin{equation}\label{supportsetsformain3}
\tilde{A}\subseteq \prod_{j=1}^d[-a_j,b_j],\  \ \tilde{B}\subseteq\prod_{j=1}^d \bR\setminus(-b_j,a_j).
\end{equation}
\end{theorem}

Theorems \ref{main1} and \ref{main3} will be proved in Sections 2 and 3, respectively. We shall then present examples and explanation to help explore the insight of these results in Section 4.

\section{A Geometric Characterization}
\setcounter{equation}{0}
We shall prove in this section Theorem \ref{main1}, which provides a geometric characterization of the Bedrosian identity.

Let $m\in L^\infty(\bR^d)$ and let $T$ be the associated bounded linear operator on $L^2(\bR^d)$ defined by (\ref{fouriermultiplier}). Always denote by $A$ and $B$ two open subsets in $\bR^d$ that come with the associated connected components and characteristic sets given by (\ref{components}) and (\ref{components2}). The Lebesgue measure is denoted as $|\cdot|$. The terms ``almost everywhere" and ``almost every" with respect to the Lebesgue measure will both be abbreviated as ``a.e.".

We start with a necessary and sufficient condition of (\ref{bedrosianidentityT}) followed from directly applying the Fourier transform to both sides of the identity.

\begin{lemma}\label{mainllemma}
The operator $T$ with the Fourier multiplier $m\in L^\infty(\bR^d)$ satisfies the Bedrosian identity with respect to $(A,B)$ if and only if for a.e. $\xi\in\bR^d$,
\begin{equation}\label{main1lemmaeq}
m(\eta)=m(\xi)\mbox{ a.e. }\eta\in (\xi-A)\cap B.
\end{equation}
\end{lemma}
\begin{proof}
Let $f,g$ be two arbitrary functions in $L^2(\bR^d)\cap L^\infty(\bR^d)$ such that $\supp\hat{f}\subseteq A$ and $\supp\hat{g}\subseteq B$. By applying the Fourier transform to both sides of the identity (\ref{bedrosianidentityT}), we get that it holds true if and only if for a.e. $\xi\in\bR^d$
\begin{equation}\label{main1lemmaeq1}
\int_{\bR^d}\hat{f}(\xi-\eta)\hat{g}(\eta)(m(\xi)-m(\eta))d\eta=0.
\end{equation}
Let $\xi\in\bR^d$ be fixed and set $\Lambda:=(\xi-A)\cap B$. Taking into account the supports of $\hat{f}$ and $\hat{g}$, one sees that equation (\ref{main1lemmaeq1}) is equivalent to
\begin{equation}\label{main1lemmaeq2}
\int_{\Lambda}\hat{f}(\xi-\eta)\hat{g}(\eta)(m(\xi)-m(\eta))d\eta=0.
\end{equation}
Clearly, if (\ref{main1lemmaeq}) is satisfied then the above equation holds true. Conversely, suppose (\ref{main1lemmaeq2}) is true for a.e. $\xi\in\bR^d$ and for all $f,g\in L^2(\bR^d)\cap L^\infty(\bR^d)$ with $\supp\hat{f}\subseteq A$ and $\supp\hat{g}\subseteq B$. Let $\xi\in\bR^d$ be such a point. We then let $\hat{f}:=\chi_{\xi-E}$ and $\hat{g}:=\chi_E$ in (\ref{main1lemmaeq2}), where $E$ is an arbitrary Lebesgue measurable subset in $\Lambda$ with $|E|<+\infty$ and $\chi_E$ denotes its characteristic function. The substitution yields
$$
\int_E (m(\xi)-m(\eta))d\eta=0,
$$
which immediately implies $m(\xi)=m(\eta)$ for a.e. $\eta\in \Lambda=(\xi-A)\cap B$.
\end{proof}

Another simple fact we shall need is as follows.

\begin{lemma}\label{constant}
Let $\Lambda$ be a connected open subset in $\bR^d$. A Lebesgue measurable function equals a constant a.e. on $\Lambda$ if and only if every $\xi\in \Lambda$ has an open neighborhood on which the function is a.e. constant.
\end{lemma}
\begin{proof}
It suffices to point out that a connected open set in $\bR^d$ is path-connected (see \cite{Gamelin}, page 90).
\end{proof}

We are ready to present the promised proof. In the proof, we shall use the notation $\bB_r(\zeta):=\{x\in \bR^d: \|x-\zeta\|<r\}$ for $r>0$ and $\zeta\in\bR^d$.\newline

{\bf Proof of Theorem \ref{main1}:} We begin with the sufficiency. Suppose that on each characteristic set of $(A,B)$, the Fourier multiplier a.e. equals a constant. We need to show that for a.e. $\xi\in\bR^d$, (\ref{main1lemmaeq}) holds true. Note that since $A,B$ are open, $(\xi-A)\cap B$ is of positive Lebesgue measure if and only if it is nonempty, which happens exactly when $\xi\in A+B$. Thus, if $\xi\notin A+B$ then (\ref{main1lemmaeq}) is automatically true. Let $\xi\in A+B$. Recall the decompositions in (\ref{components}). The point $\xi$ must lie in some connected component $C_k$ of $A+B$. We also have
$$
(\xi-A)\cap B=\bigcup_{i\in I}\bigcup_{j\in J}(\xi-A_i)\cap B_j.
$$
If $ (\xi-A_i)\cap B_j$ is nonempty then $(A_i+B_j)\cap C_k$ contains $\xi$ and hence is nonempty. Since $C_k$ is a connected component of $A+B$, we get for such $A_i$ and $B_j$ that $A_i+B_j\subseteq C_k$, and thus $B_j\subseteq D_k$ by the definitions (\ref{components1}) and (\ref{components2}). The conclusion is that for $\xi\in C_k\subseteq D_k$, $(\xi-A)\cap B\subseteq D_k$. As $m$ is a.e. constant on each $D_k$, (\ref{main1lemmaeq}) holds true for a.e. $\xi\in \bR^d$. The proof for the sufficiency part is complete.

Turning to the necessity, we assume that (\ref{main1lemmaeq}) holds true for a.e. $\xi\in \bR^d$. Our objective is to show that $m$ is a.e. constant on each $D_k$. To this end, we shall first prove that $m$ is a.e. constant on $B_j$ for each $j\in J_k$. Let $j\in J_k$. Then there exists some $i\in I$ such that $A_i+B_j\subseteq C_k$. Set
$$
U:=\{\xi\in A_i+B_j:\ \mbox{(\ref{main1lemmaeq}) holds true}\}.
$$
By our assumption, $|(A_i+B_j)\setminus U|=0$. Set
$$
\check{B_j}:=\bigcup_{\xi\in U}(\xi-A_i)\cap B_j.
$$
We claim that $\check{B_j}=B_j$. Note that
$$
B_j=\bigcup_{\xi\in A_i+B_j}(\xi-A_i)\cap B_j.
$$
As $U\subseteq (A_i+B_j)$, $\check{B_j}\subseteq B_j$. Let $\eta$ be a fixed point in $B_j$. Then there exists some $\xi\in A_i+B_j$ such that $\eta\in (\xi-A_i)\cap B_j$. Since $A_i,B_j$ are open, $\bB_r(\eta)\subseteq (\xi-A_i)\cap B_j$ for some $r>0$. By $|(A_i+B_j)\setminus U|=0$, we can find some point $\xi'\in \bB_{\frac r2}(\xi)\cap U$. A critical observation is that $ \bB_r(\eta)\subseteq (\xi-A_i)$ and $\|\xi-\xi'\|<\frac r2$ imply $\bB_{\frac r2}(\eta)\subseteq (\xi'-A_i)$. As a result, $\eta\in\bB_{\frac r2}(\eta)\subseteq (\xi'-A_i)\cap B_j$. We have hence shown $\check{B_j}=B_j$. It follows that for each $\eta\in B_j$, there exists some point $\xi\in U$ such that $\eta \in (\xi-A_i)\cap B_j\subseteq (\xi-A)\cap B$. Since $\xi\in U$ satisfies (\ref{main1lemmaeq}), $m$ is a.e. constant on $(\xi-A_i)\cap B_j$, which is an open neighborhood of $\eta$. By Lemma \ref{constant}, $m$ a.e. equals a constant $u_j$ on $B_j$.

To continue, let $i\in I$ be such that $A_i+B_j\subseteq C_k$. Recall that for a.e. $\xi\in A_i+B_j$, (\ref{main1lemmaeq}) is true, that is $m(\xi)=m(\eta)$ for a.e. $\eta\in (\xi-A_i)\cap B_j\subseteq B_j$. Therefore, $m$ is a.e. equal to $u_j$ on $A_i+B_j$. Since $C_k$ is the union of certain $A_i+B_j$ such that it remains connected, we obtain that $m$ is a.e. constant on $C_k$. Note that we have just shown that for all $j\in J_k$ and $i\in I$ with $A_i+B_j\subseteq C_k$, $m=u_j$ a.e. on both $A_i+B_j$ and $B_j$. Therefore, $m$ is a.e. constant on each characteristic set $D_k=C_k\cup\bigcup_{j\in J_k}B_j$. $\Box$
\section{Support Sets for the Partial Hilbert Transforms}
\setcounter{equation}{0}

We shall characterize in this section the support sets for the partial Hilbert transforms (\ref{partialHilbert}), and hence give proof for Theorem \ref{main3}. Toward this purpose, we first make several simple observations.

\begin{lemma}\label{compositions}
Let $A,B$ be two Lebesgue measurable subsets in $\bR^d$. If each partial Hilbert transform $H_j$, $1\le j\le d$ satisfies the Bedrosian identity with respect to $(A,B)$ then so does any linear combination of the compositions of $H_j$, $1\le j\le d$.
\end{lemma}
\begin{proof}
The result is proved by the straightforward fact that if $T_1,T_2$ satisfy the Bedrosian identity with respect to $(A,B)$ then so do $T_1T_2$ and $T_1+T_2$.
\end{proof}

To present the next one, we associate each sign vector $\nu\in\{-1,1\}^d$ with an open hyper-quadrant $Q_{\nu}$ in $\bR^d$ given by
$$
Q_{\nu}:=\{\xi\in\bR^d: \nu_j\xi_j>0,\ 1\le j\le d \}.
$$
The following observation has been made, for example, in \cite{Bulow,Zhang}.

\begin{lemma}\label{partialfourier}
A bounded linear translation-invariant operator $T$ on $L^2(\bR^d)$ is a linear combination of the compositions of $H_j$, $1\le j\le d$ if and only if its Fourier multiplier $m$ satisfies
\begin{equation}\label{multiplierofpartial}
m(\xi)=c_\nu\mbox{ for a.e. }\xi\in Q_{\nu},\ \ \nu\in\{-1,1\}^d,
\end{equation}
where $c_\nu$'s are constants.
\end{lemma}
%\begin{proof}
%The Fourier multiplier of each partial Hilbert transform satisfies (\ref{multiplierofpartial}). Thus, a linear combination of the compositions of $H_j$, $1\le j\le d$ also has a Fourier multiplier of the form (\ref{multiplierofpartial}). To prove the converse, it suffices to show that for each $\nu\in\{-1,1\}^d$, the operator $T$ with the following multiplier
%$$
%m(\xi)=\left\{
%\begin{array}{ll}
%1, &\mbox{a.e. }\xi\in Q_{\nu},\\
%0, &\mbox{elsewhere},
%\end{array}
%\right.
%$$
%is a linear combination of certain compositions of the partial Hilbert transforms. We note that $m$ above has the equivalent form
%$$
%m(\xi)=\prod_{j=1}^d \frac{(1+\nu_j\sgn(\xi_j))}2,\ \ \xi\in\bR^d.
%$$
%By the definition (\ref{partialHilbert}) of the partial Hilbert transforms, the associated operator $T$ is
%$$
%T=\prod_{j=1}^d \frac{(1+i\nu_jH_j)}2,
%$$
%which completes the proof.
%\end{proof}

The above two preparations lead to the following key lemma.

\begin{lemma}\label{keyformain3}
Let $A,B$ be two open subsets in $\bR^d$. Then all partial Hilbert transforms $H_j$, $1\le j\le d$, satisfy the Bedrosian identity with respect to $(A,B)$ if and only if each characteristic set of $(A,B)$ is contained in a certain open hyper-quadrant of $\bR^d$.
\end{lemma}
\begin{proof}
Suppose first that each characteristic set of $(A,B)$ is contained in a single open hyper-quadrant of $\bR^d$. Note that the Fourier multiplier of each $H_j$ is constant on each open hyper-quadrant. By Theorem \ref{main1}, $H_j$ satisfies the Bedrosian identity with respect to $(A,B)$. Conversely, suppose each $H_j$, $1\le j\le d$, satisfies the Bedrosian identity with respect to $(A,B)$. Then by Lemma \ref{compositions}, any operator $T$ that is a linear combination of certain compositions of the partial Hilbert transforms satisfies the Bedrosian identity with respect to $(A,B)$. By Theorem \ref{main1}, the Fourier multiplier of $T$ must be a.e. constant on each characteristic set of $(A,B)$. It follows by Lemma \ref{partialfourier} that each characteristic set of $(A,B)$ must be contained by a single open hyper-quadrant of $\bR^d$.
\end{proof}

When the support sets $A,B$ are both open, we make use of Theorem \ref{main1} to prove Theorem \ref{main3}.\newline

\noindent {\bf Proof:} Let $A,B\subseteq\bR^d$ be nonempty and open. If condition (\ref{supportsetsformain3}) is satisfied then it has been proved in \cite{Stark,VenouziouandZhang,Zhang} that all $H_j$, $1\le j\le d$, satisfy the Bedrosian identity with respect to $(A,B)$. We can now deduce this fact immediately from Lemma \ref{keyformain3} as every characteristic set of $(A,B)$ is clearly contained in a single open hyper-quadrant.

Conversely, suppose each $H_j$ for $1\le j\le d$ satisfies the Bedrosian identity with respect to $(A,B)$. Then by Lemma \ref{keyformain3}, every characteristic set of $(A,B)$ is contained in an open hyper-quadrant. For each sign vector $\nu\in\{-1,1\}^d$, set $B_\nu:=B\cap Q_\nu$. Then
\begin{equation}\label{factorB}
B=\bigcup_{\nu\in\{-1,1\}^d}B_\nu.
\end{equation}
Let $A_i$, $i\in I$ be the connected components of $A$. Assume that $B_\nu$ is nonempty. Let $\Omega$ be a connected component of $B_\nu$. Then for each $i\in I$, $\Omega\cup(A_i+\Omega)$ is contained in a characteristic set of $A+B$, and hence must be contained by a certain open hyper-quadrant in $\bR^d$. As $\Omega\subseteq B_\nu\subseteq Q_\nu$, this open hyper-quadrant must be $Q_\nu$. Thus, we have
$A_i+\Omega\subseteq Q_\nu.$ Since this inclusion holds for any connected component $A_i$ of $A$ and $\Omega$ of $B_\nu$. We obtain
\begin{equation}\label{proofformain3eq1}
A+B_\nu\subseteq Q_\nu\mbox{ for all }B_\nu\ne\emptyset.
\end{equation}
If $B_\nu$ is nonempty then set
\begin{equation}\label{proofformain3eq3}
c^\nu_j:=\inf\{\nu_j\xi_j:\xi\in B_\nu\},\ \ 1\le j\le d.
\end{equation}
Otherwise, set $c^\nu_j:=+\infty$ for $1\le j\le d$. Then equation (\ref{proofformain3eq1}) implies
$$
\{\nu_j\xi_j:\xi\in A\}\subseteq (-c^\nu_j,+\infty),\ \ 1\le j\le d.
$$
We hence obtain for each $\nu\in\{-1,1\}^d$ some nonnegative constants $0\le c^\nu_j\le+\infty$, $1\le j\le d$, such that
\begin{equation}\label{proofformain3eq2}
\nu\cdot B_\nu\subseteq \prod_{j=1}^d(c^\nu_j,+\infty),\  \ \nu\cdot A\subseteq \prod_{j=1}^d (-c^\nu_j,+\infty),
\end{equation}
where for a subset $D\subseteq \bR^d$ and sign vector $\nu\in\{-1,1\}^d$, $\nu\cdot D:=\{(\nu_1\xi_1,\nu_2\xi_2,\cdots,\nu_d\xi_d):\xi\in D\}$. For each $1\le j\le d$, let $a_j:=\min\{c^\nu_j:\nu\in\{-1,1\}^d,\,\nu_j=1\}$ and $b_j:=\min\{c^\nu_j:\nu\in\{-1,1\}^d,\,\nu_j=-1\}$. Then we have by (\ref{proofformain3eq2})
$$
A\subseteq \prod_{j=1}^d(-a_j,b_j).
$$
Let $\xi\in B$. By (\ref{factorB}), $\xi\in B_\nu$ for some $\nu\in\{-1,1\}^d$. Then by definition (\ref{proofformain3eq3}), for each $1\le j\le d$, we get $\xi_j\in (c^\nu_j,+\infty)\subseteq (a_j,+\infty)$ when $\nu_j=1$ and $\xi_j\in (-\infty, -c^\nu_j)\subseteq (-\infty,-b_j)$ when $\nu_j=-1$. Therefore, for all $\xi\in B$, $\xi_j\in \bR\setminus[-b_j,a_j]$ for all $1\le j\le d$, which implies $B\subseteq\prod_{j=1}^d \bR\setminus[-b_j,a_j]$. The proof is complete. $\Box$

\vspace{0.5cm}

In the final part of this section, we shall prove the result in Theorem \ref{main3} for general support sets $A,B$. To this end, we first recall from the introduction that each Lebesgue measurable set $E$ is associated with its {\it essential set}
$$
\tilde{E}:=\{x_0\in\bR^d: |\bB_r(x_0)\cap E|>0\mbox{ for all }r>0\}.
$$
We remark that $|E\setminus \tilde{E}|=0$. The support $\supp f$ of a Lebesgue measurable function $f$ on $\bR^d$ is the essential set of $\{x\in\bR^d:f(x)\ne0\}$. When $f$ is continuous, $\supp f=\overline{\{x\in\bR^d:f(x)\ne0\}}$ as expected.

A few more notations will be used. Denote by $\conv A$ the {\it convex hull} of a set $A\subseteq\bR^d$, which is the smallest convex set that contains $A$. Set for each $1\le j\le d$
$$
\bR^{d}_{j+}:=\{x\in\bR^{d}:x_{j}\ge 0\},\ \ \bR^{d}_{j-}:=\{x\in\bR^{d}:x_{j}\le0\}.
$$
Finally, the convolution of $\varphi,\psi\in L^2(\bR^d)$ is given by
$$
(\varphi*\psi)(x):=\int_{\bR^d} \varphi(x-y)\psi(y)dy,\ \ x\in\bR^d.
$$

We shall need the celebrated Titchmarsh convolution theorem \cite{Lions,Titchmarsh} stated below.
\begin{lemma}\label{Titchmarsh}
If  $\varphi,\psi\in L^2(\bR^d)$ are compactly supported then
$$
\conv\supp \phi\ast\psi =\conv\supp\phi +\conv\supp\psi.
$$
\end{lemma}

Another simple observation is made.

\begin{lemma}\label{main4lemma}
Let $f,g\in L^2(\bR^d)\cap L^\infty(\bR^d)$. Then $H_j(fg)=fH_jg$ for all $1\le j\le d$ if and only if for each $1\le j\le d$
\begin{equation}\label{main4lemmaeq1}
\supp \hat{f}*(\hat{g}\chi_{\bR^d_{j-}})\subseteq \bR^d_{j-}
\end{equation}
and
\begin{equation}\label{main4lemmaeq2}
\supp \hat{f}*(\hat{g}\chi_{\bR^d_{j+}})\subseteq \bR^d_{j+}.
\end{equation}
\end{lemma}
\begin{proof}
Let $1\le j\le d$. Recall that the multiplier of $H_j$ is $-i\sgn(\xi_j)$. By applying the Fourier transform to both sides of $H_j(fg)=fH_jg$, we see that the identity holds if and only if
$$
\int_{\bR^d}\hat{f}(\xi-\eta)\hat{g}(\eta)(\sgn(\xi_j)-\sgn(\eta_j))d\eta=0, \ \ \ae \ \xi\in\bR^d.
$$
The above equation is equivalent to (\ref{main4lemmaeq1}) and (\ref{main4lemmaeq2}).
\end{proof}

%\begin{theorem}\label{main4}
%Let $A,B$ be two nonempty Lebesgue measurable sets in $\bR^d$. Then all partial Hilbert transforms $H_j$, $1\le j\le d$ satisfy the Bedrosian identity with respect to $(A,B)$ if and only if there exist constants $0\le a_j,b_j\le+\infty$, $1\le j\le d$ such that
%\begin{equation}\label{supportsetsformain4}
%\tilde{A}\subseteq \prod_{j=1}^d[-a_j,b_j],\  \ \tilde{B}\subseteq\prod_{j=1}^d \bR\setminus(-b_j,a_j).
%\end{equation}
%\end{theorem}
We now prove Theorem \ref{main3} by the Titchmarsh convolution theorem.\newline

\noindent {\bf Proof of Theorem \ref{main3}: } The sufficiency has been proved in \cite{Stark,VenouziouandZhang,Zhang}. We need to deal with the necessity here. Let $A,B\subseteq\bR^d$ be Lebesgue measurable. Suppose that all partial Hilbert transforms $H_j$, $1\le j\le d$ satisfy the Bedrosian identity with respect to $(A,B)$. By Lemma \ref{main4lemma}, equations (\ref{main4lemmaeq1}) and (\ref{main4lemmaeq2}) hold true for each $1\le j\le d$ and for all $f,g\in L^2(\bR^d)\cap L^{\infty}(\bR^d)$ with $\supp \hat{f}\subseteq A$, $\supp \hat{g}\subseteq B$. We claim that it must follow that for all $1\le j\le d$
\begin{equation}\label{main4key1}
\xi_j+\eta_j\le 0 \mbox{ for all }\xi\in \tilde{A},\ \eta\in \tilde{B}_{j-}
\end{equation}
and
\begin{equation}\label{main4key2}
\xi_j+\eta_j\ge 0 \mbox{ for all }\xi\in \tilde{A},\ \eta\in \tilde{B}_{j+}.
\end{equation}
Here, $\tilde{B}_{j-}:=\tilde{B}\cap \bR^d_{j-}$ and $\tilde{B}_{j+}:=\tilde{B}\cap \bR^d_{j+}$.

Let $1\le j\le d$. We shall prove (\ref{main4key1}) by contradiction. The other claim (\ref{main4key2}) can be verified in a similar manner. Assume to the contrary that for some $\xi'\in \tilde{A}$ and $\eta'\in\tilde{B}_{j-}$, $\xi'_j+\eta_j'>0$. Then there exist open balls $U,V$ of sufficiently small radius such that $\xi'\in U, \eta'\in V$, $|A\cap U|>0$, $|B\cap V|>0$, and $V\subseteq \bR^d_{j-}$. Introduce
$$
\hat{f}(\xi):=e^{-\|\xi\|^2}\chi_{A\cap U}(\xi),\ \hat{g}(\xi):=e^{-\|\xi\|^2}\chi_{B\cap V}(\xi), \ \ \xi\in\bR^d.
$$
Note that $f,g$ defined above do lie in $L^2(\bR^d)\cap L^\infty(\bR^d)$. Now since $\xi'\in\supp\hat{f}$, $\eta'\in\supp\hat{g}$, and $\xi'_j+\eta'_j>0$,
$$
\supp \hat{f}+ \supp \hat{g}\nsubseteq \bR^d_{j-}.
$$
As $\hat{f}, \hat{g}$ are of compact support, we get by Lemma \ref{Titchmarsh}
$$
\conv\supp \hat{f}*\hat{g}= \conv\supp \hat{f}+ \conv\supp \hat{g} \nsubseteq \bR^d_{j-}.
$$
Consequently, $\supp \hat{f}*(\hat{g}\chi_{\bR^d_{j-}})=\supp\hat{f}*\hat{g} \nsubseteq \bR^d_{j-}$, contradicting (\ref{main4lemmaeq1}).

We have hence proved (\ref{main4key1}) and (\ref{main4key2}) for all $1\le j\le d$. Set for each $1\le j\le d$,
$$
b_j:=\inf\{-\eta_j: \eta\in \tilde{B}_{j-}\}
$$
if $\tilde{B}_{j-}$ is nonempty. Otherwise, set $b_j:=+\infty$. Equation (\ref{main4key1}) then implies that $\tilde{A}\subseteq\{\xi\in\bR^d:\xi_j\le b_j\}$ for each $1\le j\le d$. Also, set
$$
a_j:=\inf\{\eta_j: \eta\in \tilde{B}_{j+}\}
$$
if $\tilde{B}_{j+}$ is nonempty. Otherwise, set $a_j:=+\infty$. It follows from (\ref{main4key2}) that $\tilde{A}\subseteq\{\xi\in\bR^d:\xi_j\ge -a_j\}$ for each $1\le j\le d$. The proof is complete by noticing that the set inclusions in (\ref{supportsetsformain3}) are satisfied. $\quad\Box$

\section{Examples}
\setcounter{equation}{0}

We present a few examples to illustrate Theorem \ref{main1} and Corollary \ref{main2}. For brevity, we shall call an operator $T$ Bedrosian with $(A,B)$ to mean that it satisfies the Bedrosian identity with respect to $(A,B)$. Furthermore, ``bounded linear translation-invariant" will be abbreviated as ``BLTI".

\begin{description}
\item[{\bf Example 4.1} (Balls)] Fixed $r>0$ and $\zeta\in\bR^d$. Let $A:=\bB_r(\zeta)$ and $B:=\{x\in\bR^d:\|x+\zeta\|>r\}$. Then $A+B=\bR^d\setminus\{0\}$. There are three cases.
    \begin{enumerate}[(i)]
    \item $\|\zeta\|>r$. Then $0\in B$. As a result, we have the whole space $\bR^d$ as the only characteristic set. By Corollary \ref{main2}, there does not exist a nontrivial BLTI operator on $L^2(\bR^d)$ that is Bedrosian with $(A,B)$.
    \item $\|\zeta\|\le r$ and $d=1$. In this case, we have two characteristic sets $(0,+\infty)$ and $(-\infty,0)$. By Theorem \ref{main1}, $T$ is Bedrosian with $A$ and $B$ if and only if its multiplier has the form
    $$
    m(\xi):=\left\{
    \begin{array}{ll}
    c_1,&\mbox{for almost every }\xi>0,\\
    c_2,&\mbox{for almost every }\xi<0,
    \end{array}
    \right.
    $$
    where $c_1$ and $c_2$ are constants. One observes that the Fourier multiplier of $T$ is given above if and only if $T$ is a linear combination of the Hilbert transform and the identity operator.
    \item $\|\zeta\|\le r$ and $d\ge 2$. In this case, $\bR^d\setminus\{0\}$ is connected and hence is the only characteristic set of $(A,B)$. By Corollary \ref{main2}, we do not have a nontrivial BLTI operator on $L^2(\bR^d)$ that is Bedrosian with $(A,B)$.

    We make remarks about cases (ii) and (iii) above. It has been obtained in a different way in \cite{Zhang} that when $d\ge2$, there is not a nontrivial BLTI operator on $L^2(\bR^d)$ that is Bedrosian with $A:=\bB_r(0)$ and $B:=\{x\in\bR^d:\|x\|>r\}$ for all $r>0$. We now know that a single pair of such support sets is enough to force the nonexistence for $d\ge 2$, and understand it is connectedness that causes the fundamental difference between $d=1$ and $d\ge 2$. Case (iii) is plotted below.
    \begin{figure}[htbp]
    \centering
    \includegraphics[width=0.4\textheight]{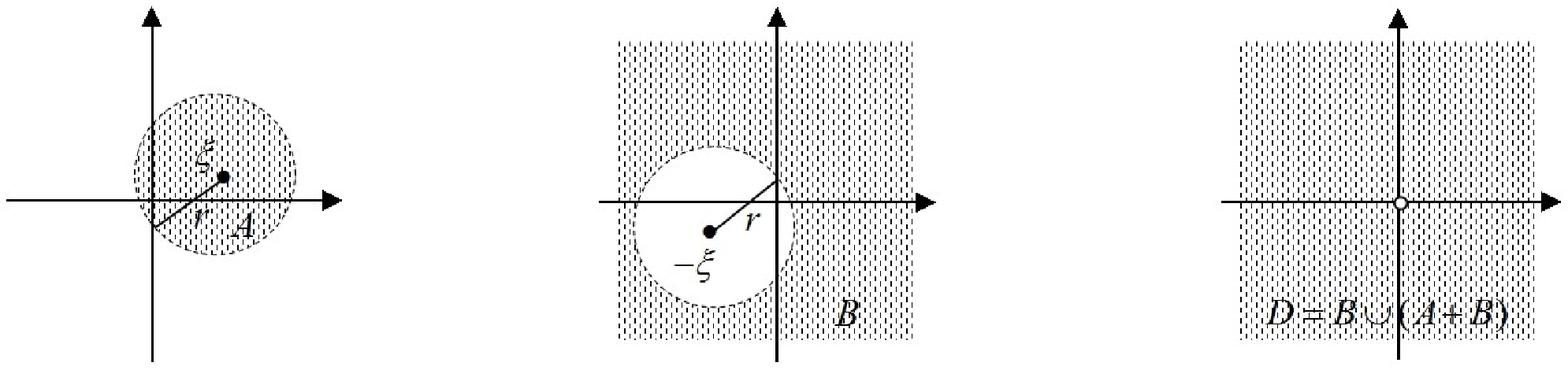}\\
    % \caption{}\label{}
    \end{figure}
    \end{enumerate}

\item[{\bf Example 4.2 (Modified balls)}]  Let $d\ge2$, $r,\varepsilon_0$ be two positive constants, and $\zeta\in\bR^d$ with $\|\zeta\|\le r$. Set $A:=\bB_r(\zeta)$ and $B:=\{x\in\bR^d:\|x+\zeta\|>r+\varepsilon_0\}$. Then there is only one characteristic set of $(A,B)$, which is $\{x\in\bR^d:\|x\|>\varepsilon_0\}$. By Corollary \ref{main2}, there are nontrivial BLTI operators on $L^2(\bR^d)$ that are Bedrosian with $(A,B)$. By Theorem \ref{main1}, such operators must have a Fourier multiplier of the following form
    $$
    m(\xi):=\left\{
    \begin{array}{ll}
    c,&\mbox{for almost every }\|\xi\|>\varepsilon_0,\\
    \varphi(\xi),&\mbox{for }\|\xi\|< \varepsilon_0,
    \end{array}
    \right.
    $$
    where $c$ is a constant and $\varphi$ is an arbitrary function in $L^\infty(\bB_{\varepsilon_0}(0))$. The difference from Example 4.1 is that by enlarging the distance of $B$ to its center a little bit, one gets nontrivial operators that satisfy the Bedrosian identity. The support sets and the characteristic set in this case are illustrated below.

%We illustrate the supports of $\hat{f}$ and $\hat{g}$ in Example 2.1 and Example 2.2 for the two-dimensional case in the following graph.

\begin{figure}[htbp]
  \centering
  \includegraphics[width=0.4\textheight]{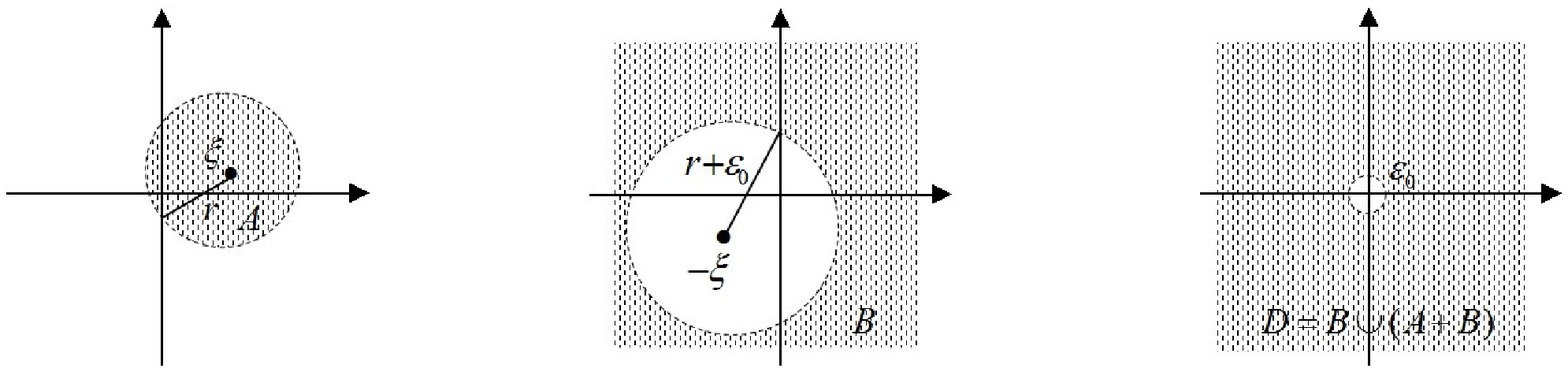}\\
 % \caption{}\label{}
\end{figure}

\item[{\bf Example 4.3 (Rectangular support sets)}]  Let $a_j, b_j$ be positive constants, $1\le j\le d$. Set
    \begin{equation}\label{boxes}
    A:=\prod_{j=1}^d (-a_j,b_j),\ \ B:=\prod_{j=1}^d \bR\setminus [-b_j,a_j].
    \end{equation}
    In this case, $(A,B)$ has $2^d$ characteristic sets, which are exactly the $2^d$ open hyper-quadrants in $\bR^d$. By Theorem \ref{main1}, $T$ is Bedrosian with $(A,B)$ if and only if on each open hyper-quadrant, the Fourier multiplier of $T$ is a.e. constant, which is also equivalent to that $T$ is a linear combination of the compositions of the partial Hilbert transforms on $\bR^d$.

    We mention that it was proved in \cite{VenouziouandZhang} that a BLTI operator $T$ on $L^2(\bR^d)$ is Bedrosian with $(A,B)$ given by (\ref{boxes}) for all $a_j,b_j>0$, $1\le j\le d$, if and only if it is a linear combination of the compositions of the partial Hilbert transforms. The improvement of our result here is again that we show a single pair of support sets of the form (\ref{boxes}) is enough to ensure that the corresponding operator $T$ must be given as described above. We illustrate this example below.
\begin{figure}[htbp]
  \centering
  \includegraphics[width=0.5\textheight]{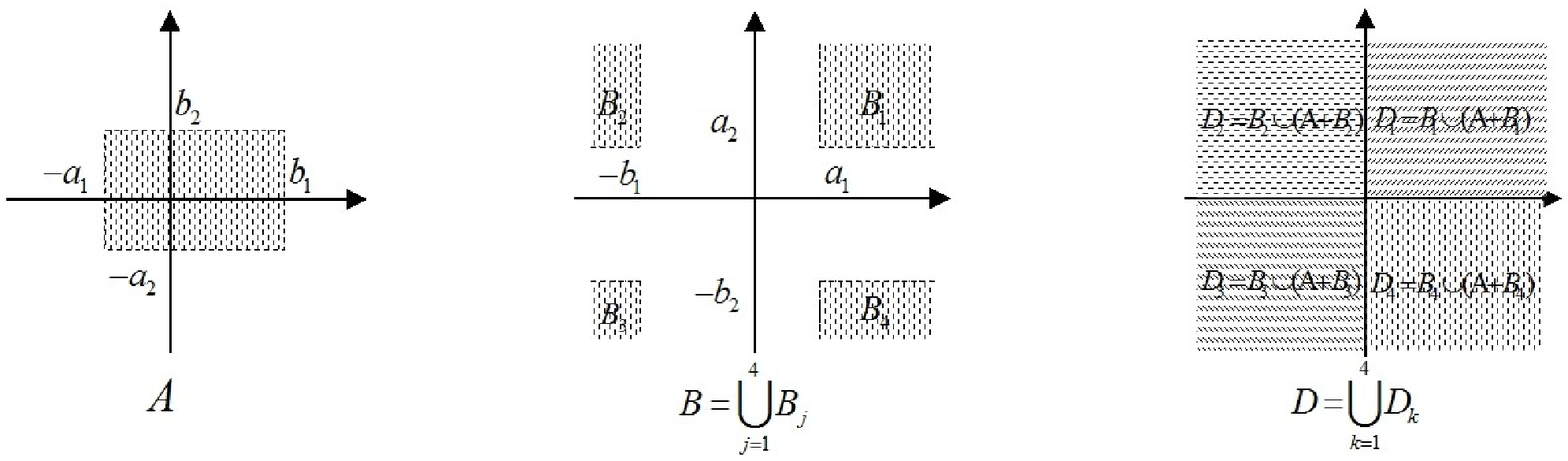}\\
 % \caption{}\label{}
\end{figure}

\item[{\bf Example 4.4} (Hyper-quadrants)]  Let $\mu,\nu\in \{-1,1\}^d$ and set $A:=Q_\mu$, $B:=Q_\nu$. There are two cases.
    \begin{enumerate}[(i)]
    \item $-A=B$. Then $A+B=\bR^d$. Thus, there does not exist a nontrivial BLTI operator that is Bedrosian with $(A,B)$.

    \item $-A\ne B$. Let $\bI:=\{j: 1\le j\le d, \mu_j=\nu_j\}$. Then we have only one characteristic set, which is $$
        (A+B)\cup B=\{\xi\in\bR^d: \mu_j\xi_j>0,\ j\in \bI\}.
        $$
     In this case, we have nontrivial BLTI operators that are Bedrosian with $(A,B)$. By Theorem \ref{main1}, the Fourier multipliers of such operators are a.e. constant on the characteristic set above. The support sets and the characteristic set in this case are illustrated below.
    \begin{figure}[htbp]
    \centering
    \includegraphics[width=0.4\textheight]{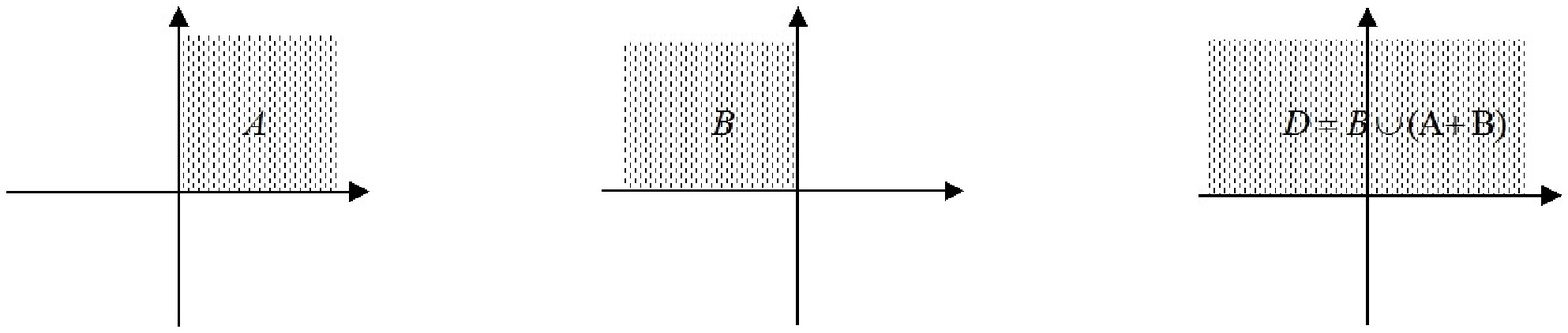}\\
    % \caption{}\label{}
    \end{figure}

     \end{enumerate}

\item[{\bf Example 4.5} (Modified rectangles)]  The support sets $A$ and $B$ are interpreted as respectively containing the low and high Fourier frequency (see, for example, \cite{Picibono}). What will happen if we interchange their roles? Let $a_j, b_j$ be positive constants, $1\le j\le d$. Set
    \begin{equation}
    A:=\prod_{j=1}^d \bR\setminus [-b_j,a_j],\ \ B:=\prod_{j=1}^d (-a_j,b_j).
    \end{equation}
    Then $A+B$ has the whole space $\bR^d$ as its only characteristic set. Consequently, there does not exist a nontrivial BLTI operator on $L^2(\bR^d)$ that is Bedrosian with $(A,B)$.

\item[{\bf Example 4.6} (Bounded support sets)]  If $A$ and $B$ are both bounded open subsets in $\bR^d$ then all characteristic sets of $(A,B)$ are bounded. Therefore, there must exist nontrivial BLTI operators on $L^2(\bR^d)$ that are Bedrosian with $(A,B)$.

%\item[{\bf Example 2.7}]  Let $\theta'_j,\theta''_j\in [0,\pi]$, $1\le j\le d-2$, and let $\theta''_{d-1},\theta''_{d-1}\in [0,2\pi]$, $\alpha,\beta\in [0,\frac\pi2]$. We choose two nonempty infinitely unilateral cones $A$ and $B$ in $\bR^d$ with the apex points at the origin, defined by
%    $$
%    A:=\{x\in\bR^d: \arccos (\frac{x}{\|x\|}\cdot \nu_A)<\alpha\},\
%    B:=\{x\in\bR^d: \arccos (\frac{x}{\|x\|}\cdot \nu_B)<\beta\}.
%    $$
%    Here, the directions $\nu_A,\nu_B$ of the axes of cones $A,B$ are  respectively given as
%    $$
%    \nu_A:=(\cos\theta'_1,\sin\theta'_1\cos\theta'_2,\ldots,\sin\theta'_1\cdots\sin\theta'_{d-2}\cos\theta'_{d-1},
%    \sin\theta'_1\cdots\sin\theta'_{d-2}\sin\theta'_{d-1} ),
%    $$
%    $$
%    \nu_B:=(\cos\theta''_1, \sin\theta''_1\cos\theta''_2,\ldots,\sin\theta''_1\cdots\sin\theta''_{d-2}\cos\theta''_{d-1},
%    \sin\theta''_1\cdots\sin\theta''_{d-2}\sin\theta''_{d-1} ).
    %$$

    %For sake of clarity, we illustrate two unilateral cones $A,B$ for the two-dimensional case in the following graph.
%    \begin{figure}[htbp]
%    \centering
%    \includegraphics[width=0.3\textheight]{cones.eps}\\
%    %\caption{Two  infinitely unilateral cones $A,B$ in $\bR^2$.}
%    \end{figure}
\end{description}

{\small
\bibliographystyle{amsplain}

\begin{thebibliography}{30}
\bibitem{Bedrosian} E. Bedrosian, A product theorem for Hilbert transforms, {\it Proc. IEEE} \textbf{51} (1963), 868--869.
\bibitem{Brown74} J. L. Brown, Analytic signals and product theorems for Hilbert transforms, {\it IEEE Trans. Circuits and Systems} \textbf{21} (1974), 790--792.
\bibitem{Brown86} J. L. Brown, A Hilbert transform product theorem, {\it Proc. IEEE} \text{74} (1986), 520--521.
\bibitem{Bulow} T. Bulow and G. Sommer, Hypercomplex signals--a novel extension of the analytic signal to the multidimensional case, {\it IEEE Trans. Signal Process.} \textbf{49} (2001), 2844--2852.

\bibitem{Butzer} P. L. Butzer and R. J. Nessel, {\it Fourier Analysis and
Approximation. Volume 1: One-dimensional Theory}, Pure and Applied
Mathematics, Vol. 40, Academic Press, New York, 1971.
%\bibitem{Conway} J. B. Conway, {\it A Course in Functional Analysis}, 2nd Edition, Springer-Verlag, New York, 1990.
\bibitem{Cohen} L. Cohen, {\it Time-Frequency Analysis: Theory and Applications}, Prentice Hall, Englewood Cliffs, NJ, 1995.
%\bibitem{Gasquet} C. Gasquet and P. Witomski, {\it Fourier Analysis and Applications}, Springer-Verlag, New %York, 1999.
\bibitem{Gamelin} T. W. Gamelin and R. E. Greene, {\it Introduction to Topology}, Second edition, Dover Publications, Inc., Mineola, NY, 1999.
\bibitem{Grafakos} L. Grafakos, {\it Classical and Modern Fourier Analysis}, Prentice Hall, New Jersey, 2004.
\bibitem{Hahn} S. L. Hahn, Multidimensional complex signals with single-orthant spectra, {\it Proc. IEEE} \textbf{80} (1992), 1287--1300.
\bibitem{Huang} N. E. Huang, et al., The empirical mode decomposition and the Hilbert spectrum for nonlinear and non-stationary time series analysis, {\it R. Soc. Lond. Proc. Ser. A Math. Phys. Eng. Sci.} \textbf{454} (1998), 903--995.

\bibitem{Lions} J. L. Lions, Support dans la transformation de Laplace, {\it J. Anal. Math.} \textbf{2} (1952/53), 369--380.
%\bibitem{Munkres} J. R. Munkres, {\it Topology}, 2nd Edition, Prentice Hall, Upper Saddle River, NJ, 2000.
\bibitem{Nuttall} A. H. Nuttall and E. Bedrosian, On the quadrature approximation to the Hilbert transform of modulated signals, {\it Proc. IEEE}. \textbf{54} (1966), 1458--1459.
\bibitem{Picibono} B. Picibono, On instantaneous amplitude and phase of signals, {\it IEEE Trans. Siganl Process}. \textbf{45} (1997) 552--560.
\bibitem{Qian2009} T. Qian, Y. Xu, D. Yan, L. Yan, and B. Yu, Fourier spectrum characterization of Hardy spaces and applications, {\it Proc. Amer. Math. Soc.} \textbf{137} (2009), 971--980.
\bibitem{Stark} H. Stark, An extention of the Hilbert transform product theorem, {\it Proc. IEEE} \textbf{59} (1971), 1359--1360.
%\bibitem{Stein} E. M. Stein, {\it Harmonic Analysis}, Princeton University Press, Princeton, NJ, 1993.
\bibitem{Tan2010} L. Tan, L. Shen, and L. Yang, Rational orthogonal bases satisfying the Bedrosian identity, {\it Adv. Comput. Math.} \textbf{33} (2010), 285--303.
\bibitem{Tan2009} L. Tan, L. Yang, and D. Huang, Necessary and sufficient conditions for the Bedrosian identity, {\it J. Integral Equations Appl.} \textbf{21} (2009), 77--94.
\bibitem{Titchmarsh} E. C. Titchmarsh, The zeros of certain integral functions, {\it Proc. London Math. Soc.} \textbf{S2-25} (1926), 283-302.
\bibitem{VenouziouandZhang} M. Venouziou and H. Zhang, Charaterizing the Hilbert transform by the Bedrosian theorem, {\it J. Math. Anal. Appl.} \textbf{338} (2008), 1477--1481.
\bibitem{Wang} S. Wang, Simple proofs of the Bedrosian equality for the Hilbert transform, {\it Sci. China Ser. A} \textbf{52} (2009), 507--510.
\bibitem{Xu} Y. Xu and D. Yan, The Bedrosian identity for the Hilbert transform of product functions, {\it Proc. Amer. Math. Soc.} \textbf{134} (2006), 2719--2728.
%\bibitem{XuandZhang} Y. Xu and H. Zhang, Recent mathematical developments on empirical mode decomposition, {\it Adv. Adapt. Data Anal.} \textbf{1} (2009), 681--702.
\bibitem{YangandZhang} L. Yang and H. Zhang, The Bedrosian identity for $H^{p}$ functions, {\it J. Math. Anal. Appl.} \textbf{345} (2008), 975--984.
\bibitem{ZhangandYu} B. Yu and H. Zhang, The Bedrosian identity and homogenous semi-convolution equations, {\it J. Integral Equations Appl.} \textbf{20} (2008), 527--568.
\bibitem{Zhang} H. Zhang, Multidimensional analytic signals and the Bedrosian identity, {\it Integr. Equ. Oper. Theory} \textbf{78} (2014), 301--321.

\end{thebibliography}

}
\end{document}